\documentclass[a4paper,12pt,reqno]{article}
\usepackage[T1]{fontenc}
\usepackage{authblk}
\usepackage{epsfig,graphicx,subfigure}
\usepackage{amsthm,amsmath,amssymb,amsfonts}
\usepackage{color,xcolor}
\usepackage[all]{xy}
\usepackage[top=40mm, bottom=40mm, left=40mm, right=35mm]{geometry}
\usepackage{hyperref}

\newtheorem{theorem}{Theorem}[section]

\newtheorem{proposition}[theorem]{Proposition}
\newtheorem{corollary}[theorem]{Corollary}
\theoremstyle{definition}

\numberwithin{equation}{section}

\begin{document}
\setcounter{page}{1}
\title{\vspace{-1.5cm}
\vspace{.5cm}
\vspace{.7cm}
{\large{\bf  $\theta$-derivations on convolution algebras} }}
 \date{}
\author{{\small \vspace{-2mm} M. Eisaei$^1$ and Gh. R. Moghimi $^1$\footnote{Corresponding author} }}
\affil{\small{\vspace{-4mm}  $^1$ Department of Mathematics,
Payame Noor University (PNU), Tehran 19395-4697, Iran
}}
\affil{\small{\vspace{-4mm}  mojdehessaei59@student.pnu.ac.ir }}
\affil{\small{\vspace{-4mm} moghimimath@pnu.ac.ir}}
\maketitle
\hrule
\begin{abstract}
\noindent
In this paper, we investigate $\theta$-derivations on Banach algebra $ L_0^{\infty} (w)^*$. First, we study the range of them and prove the Singer-Wermer conjucture. We also give a characterization of the space of all $\theta$-derivations on $ L_0^{\infty} (w)^*$. Then, we prove automatic continuity and Posner's theorems for $\theta$-derivations.
 \end{abstract}

 \noindent \textbf{Keywords}: Convolution algebras, $\theta$-derivations, Singer-Wermer conjucture,  automatic continuity, Posner's theorems.\\
{\textbf{2020 MSC}}: 47B47, 46H40, 16W25.
\\
\hrule
\vspace{0.5 cm}
\baselineskip=.8cm
\section{Introduction}
Let $A$ be a Banach algebra with the center $Z(A)$ and the right annihilator; i.e.,
\begin{equation*}
Z(A)=\{ a \in A: ~~ ax=xa ~~~  \text{for all}~~~ x \in A \}
\end{equation*}
and
\begin{equation*}
\hbox{ran}(A)=\{ r \in A: ~~ ar=0 ~~~ \text{for all}~~~ a \in A\}.
\end{equation*}
Let $D: A \rightarrow A$ be a linear map and $k \in \Bbb{N}$.
Then $D$ is called \textit{$k$-centralizing}  if for every $m \in A$, we have
\begin{equation*}
[D(m), m^k ] := D(m) m^k- m^k D(m) \in Z(A).
\end{equation*}
In particular, if $[D(m), m^k]=0$, then $D$ is called
\textit{$k$-commuting}.
Assume now that $\theta: A \rightarrow A$ is a homomorphism. Then $D$ is called a $\theta$\emph{-derivation} if
\begin{equation*}
D(mn)=D(m) \theta(n) + \theta(m) D(n)
\end{equation*}
for all $m,n \in A$.
If $\theta$ is the identity map, then $D$ is called a \textit{derivation}.

Let us recall that a continuous function $w:[0,\infty) \rightarrow [1,\infty)$
is called a \textit{weight function} if $w(0)=1$ and for all $x,y \in [0, \infty)$
\begin{equation*}
w(x+y) \leq w(x) w(y).
\end{equation*}
Let $L^1(w)$ be the Banach space of all Lebesgue measurable functions $f$ on $[0, \infty)$ such that $wf \in L^1([0,\infty))$, the Banach algebra of all Lebesque integrable functions on $[0,\infty)$.
It is well-known that $L^1(w)$ is a Banach algebra with the convolution product
\begin{equation*}
\varphi * \psi(x)= \int_{0}^{\infty} \varphi(y) \psi(x-y) d(y), ~~~~ (\varphi, \psi \in L^1([0, \infty))
\end{equation*}
and the norm
\begin{equation*}
\|\varphi\|_w= \int_{0}^{\infty} w(x) |\varphi|(x) dx ~~~ (\varphi \in L^1([0, \infty));
\end{equation*}
see [4, 15].
Let also $L_0^{\infty}(w)$ be the Banach space of all Lebesgue measure functions $f$ on $[0, \infty)$ such that
\begin{equation*}
\lim_{x \rightarrow \infty} \text{ess sup}
 \bigg \{ \frac{f(y) \chi_{(x, \infty)} (y)}{w(y)}: y \geq 0 \bigg\}=0,
\end{equation*}
where $\chi_{(x, \infty)}$ is the characteristic function of $(x, \infty)$ on $[0, \infty)$.
It is well-known that the dual of $L_0^{\infty} (w)$, represented by
$L_0^{\infty} (w)^*$, is a Banach algebra with the first Arens product defined by
\begin{equation*}
\langle mn, f \rangle = \langle m, nf \rangle,
\end{equation*}
where $\langle nf, \varphi \rangle = \langle n, f \varphi \rangle$, in which
\begin{equation*}
f \varphi(x)= \int_{0}^{\infty} f(x+y) \varphi(y) dy
\end{equation*}
for all $m,n \in L_0^{\infty} (w)^*$, $f \in L_0^{\infty} (w)$, $\varphi \in L^1(w)$ and $x\geq 0$; see [7, 8, 11, 12].
Note that every element $\varphi \in L^1(w)$ can be regarded as an element of $ L_0^{\infty} (w)^*$,
\begin{equation*}
\langle \varphi, f \rangle = \int \varphi(x) f(x) dx
\end{equation*}
for all $f \in  L_0^{\infty} (w)$.
We denoted by
$\Lambda ( L_0^{\infty} (w)^*) $ the set of all right identities of
$ L_0^{\infty} (w)^*)$ with bounded one.
For every $u \in \Lambda ( L_0^{\infty} (w)^*) $ and
$m \in L_0^{\infty} (w)^*$, we have $m-um \in \hbox{ran} (L_0^{\infty} (w)^*)$ and
\begin{equation*}
m= um+(m-um).
\end{equation*}
It follows that
\begin{equation*}
 L_0^{\infty} (w)^*= u  L_0^{\infty} (w)^* \oplus  \hbox{ran} (L_0^{\infty} (w)^*).
\end{equation*}
One can prove that the radical of $ L_0^{\infty} (w)^*$ is equal to
$ \hbox{ran} (L_0^{\infty} (w)^*)$; see [13].

Derivations and $\theta$-derivations were studied by several authors [1-3, 9, 10, 13, 14].
For example, derivations on $L_0^{\infty} (w)^*$ investigated in [13].
They proved that the range of a derivation on  $L_0^{\infty} (w)^*$ is contained into $\hbox{ran}(L_0^{\infty} (w)^*)$.
They also showed that the zero map is the only $k$-centralizing derivation on $L_0^{\infty} (w)^*$.

In this paper, we investigate $\theta$-derivations  on  $L_0^{\infty} (w)^*$. In the case where, $\theta$ is an isomorphism, we prove that the range of $\theta$-derivations  on  $L_0^{\infty} (w)^*$ is contained into the radical of $L_0^{\infty} (w)^*$. If $\theta$ is also continuous, then $D$ is continuous if and only if $D|_{\hbox{ran}(L_0^{\infty} (w)^*)}$ is continuous. In this case, $D|_{uL_0^{\infty} (w)^*}$ is always continuous. 
Finally, we study Posner first and second theorems for $\theta$-derivations  on  $L_0^{\infty} (w)^*$.

\section{Main Results}
Singer and Wermer [16] showed that the range of a continuous derivation on a
commutative Banach algebra is a subset of its radical. They conjectured that the continuity requirement for the derivations can be removed.
Thomas [17]  proved the conjecture. In the sequal, we investigate this conjecture for $\theta$-derivation  on non-commutative Banach algebra $L_0^{\infty} (w)^*$.

\begin{theorem} \label{m1}
Let $\theta $ be a homomorphism on $L_0^{\infty} (w)^*$ and
$D$ be a $\theta$-derivation  on  $L_0^{\infty} (w)^*$.
Then the following statements hold.

\emph{(i)}
$D$ maps $\emph{ran} (L_0^{\infty} (w)^*)$ and $\Lambda (L_0^{\infty} (w)^*)$ into  $\emph{ran} (L_0^{\infty} (w)^*)$.

\emph{(ii)}
If $\theta$ is an isomorphism, then $D$ maps
$L_0^{\infty} (w)^*$ into  $\emph{ran} (L_0^{\infty} (w)^*)$.
\end{theorem}
\begin{proof}
$\text{(i)}$ First, note that if
$u \in \Lambda(L_0^{\infty}(w)^*)$ and $r \in \hbox{ran}(L_0^{\infty}(w)^*)$, then
$\theta(r) \in \hbox{ran}(L_0^{\infty}(w)^*)$ and
\begin{eqnarray*}
\theta(u)= u+r_0
\end{eqnarray*}
for some $r_0 \in \hbox{ran}(L_0^{\infty}(w)^*)$;
see Lemma 2.1 of [5].
So for every $k \in L_0^{\infty}(w)^*$, we have
\begin{eqnarray*}
k D(r)&=& k (u+r_0). D(r)\\
&=& k \theta(u) D(r)\\
&=& k[D(ur)-D(u) \theta(r)]=0.
\end{eqnarray*}
This shows that $D(r) \in \hbox{ran}(L_0^{\infty}(w)^*)$.
We also have
\begin{eqnarray*}
D(uu) &=&D(u) \theta (u) + \theta(u) D(u)  \\
&=& D(u) + \theta (u) D(u).
\end{eqnarray*}
Hence $\theta (u) D(u)=0$.
It follows that
\begin{eqnarray*}
kD(u) = k\theta(u) D(u)=0
\end{eqnarray*}
for all $k \in L_0^{\infty}(w)^*$.
Therefore, $D(u) \in \hbox{ran}(L_0^{\infty}(w)^*)$. \\
$\text{(ii)}$ Let $\theta$ be an isomorphism. Then $\theta^{-1}D$ is a derivation on $L_0^{\infty}(w)^*$.
By Theorem 2.1 of [5],
\begin{eqnarray*}
\theta^{-1} D(L_0^{\infty}(w)^*) \subseteq \hbox{ran}(L_0^{\infty}(w)^*).
\end{eqnarray*}
Thus
\begin{eqnarray*}
D(L_0^{\infty}(w)^*) \subseteq \theta(\hbox{ran}(L_0^{\infty}(w)^*)) \subseteq \hbox{ran}(L_0^{\infty}(w)^*),
\end{eqnarray*}
as claimed.
\end{proof}
A mapping $T: L_0^{\infty} (w)^* \rightarrow L_0^{\infty} (w)^*$ is called \emph{spectrally bounded} if there exists $c \geq 0$ such that $r(T(m)) \leq \alpha r(m)$ for all $m \in L_0^{\infty} (w)^*$, where $r(m)$ denotes the spectral radius of $m$.
\begin{corollary} \label{m2}
Let $\theta$ be a homomorphism on $L_0^{\infty} (w)^*$.  Then the following statements hold.

\emph{(i)}
The product of two $\theta$-derivations on $L_0^{\infty} (w)^*$
is a $\theta$-derivation on $\emph{ran} (L_0^{\infty} (w)^*)$.

\emph{(ii)} Every $\theta$-derivation on $\emph{ran} (L_0^{\infty} (w)^*)$ is spectrally bounded.
\end{corollary}
In the next result, we investigate the automatic continuity of $\theta$-derivation on $L_0^{\infty} (w)^*$; see [6] for the automatic continuity of derivation on
commutative semisimple Banach algebras.

\begin{proposition} \label{m3}
Let $\theta$ be a continuous isomorphism on $ L_0^{\infty} (w)^*$ and $D$ be a $\theta$-derivation on $L_0^{\infty} (w)^*$.
Then the following statements hold.

\emph{(i)}
$D|_{u L_0^{\infty} (w)^*}$ is always continuous for all $u \in \Lambda (L_0^{\infty} (w)^*)$.

\emph{(ii)}
$D$ is continuous if and only if
$D|_{ \emph{ran}(L_0^{\infty} (w)^*)}$ is continuous.
\end{proposition}
\begin{proof}
Let $\theta$ be a continuous isomorphism on $L_0^{\infty}(w)^*$ and $D$ be a $\theta$-derivation on $L_0^{\infty}(w)^*$.
Then $D$ maps $L_0^{\infty}(w)^*$ into $\hbox{ran}(L_0^{\infty}(w)^*)$ and so
\begin{eqnarray*}
D(um)&=& D(u) \theta(m) + \theta(u) D(m)\\
&=& D(u) \theta(m)= D(u) \theta(u) \theta(m)\\
&=& D(u) \theta (um)
\end{eqnarray*}
for all $m \in L_0^{\infty}(w)^*$ and $u \in \Lambda ( L_0^{\infty}(w)^*)$.
Thus
\begin{eqnarray*}
\|D(um) \|&=& \|D(u) \theta(um) \|\\
& \leq &\|D(u)\| \|\theta (um) \| \\ 
&\leq& \|D(u) \| \|\theta\| \| um\|.
\end{eqnarray*}
For (ii), let $D$ be continuous on $\hbox{ran}(L_0^{\infty}(w)^*)$.
Then there exist $c_1, c_2 \geq 0$
 such that for every $m \in L_0^{\infty}(w)^*$,
$u \in \Lambda(L_0^{\infty}(w)^*)$  and $r \in \hbox{ran}(L_0^{\infty}(w)^*)$,
\begin{equation*}
\|D(r) \| \leq c_1 \|r\| \quad \text{and} \quad \|D(um) \| \leq c_2 \|um\|.
\end{equation*}
Assume that $m \in L_0^{\infty}(w)^*$. Then
\begin{equation*}
m=um+r,
\end{equation*}
where
$r=m-um \in \hbox{ran}(L_0^{\infty}(w)^*)$. So by (i), we obtain
\begin{eqnarray*}
\|D(m)\|&=& \|D(um)+ D(r)\| \\
&\leq& \|D(um)\| + \|D(r)\| \\
&\leq& c_2 \|um\| + c_1 \|r\| \\
&\leq& c_2 \|u\| \|m\| + c_1 ( \|m\|+ \|u\| \|m\|) \\
&=& (c_2 + 2 c_1) \|m\|.
\end{eqnarray*}
That is, $D$ is continuous. The converse is clear.
\end{proof}
Let $\mathcal{A}$  be a Banach algebra and $A$  be a closed subalgebra of
$\mathcal{A}$.
 We denote by $\emph{Der}(\mathcal{A}, A)$
the space of all $\theta$-derivations from
  $\mathcal{A}$  into $A$,
where $\theta$ is a homomorphism on  $\mathcal{A}$.
We also denote by
 $\mathcal{B}(\mathcal{A}, A)$ the space of all bounded linear operators from  $\mathcal{A}$  into $A$.
We write
 $\emph{Der}(\mathcal{A}):= \emph{Der}(\mathcal{A}, A)$ and $\mathcal{B}(\mathcal{A}):= \mathcal{B}(\mathcal{A}, A)$.
\begin{theorem} \label{m4}
Let $\theta$ be an isometrically isomorphism on
$L_0^{\infty}(w)^*$ and $u \in \Lambda (L_0^{\infty}(w)^*)$. Then
\begin{equation*}
\emph{Der}(L_0^{\infty}(w)^*)= \emph{Der}(L_0^{\infty}(w)^*, u L_0^{\infty}(w)^*) \oplus \emph{Der}(L_0^{\infty}(w)^*, \emph{ran}(L_0^{\infty}(w)^*))
\end{equation*}
\end{theorem}
\begin{proof}
Let $D \in \emph{Der}(L_0^{\infty}(w)^*)$.
We define $d: L_0^{\infty}(w)^* \rightarrow u L_0^{\infty}(w)^*$
and $T: L_0^{\infty}(w)^* \rightarrow L_0^{\infty}(w)^*$ by
\begin{equation*}
d(m)= D(um) \quad \text{and} \quad T(m)= D(m- um)
\end{equation*}
for all $m \in L_0^{\infty}(w)^*$. Note that
\begin{eqnarray*}
D(m)&=& D(um+(m-um))\\
& =& D(um)+D(m-um)\\
&=&d(m)+ T(m)
\end{eqnarray*}
for all $m \in L_0^{\infty}(w)^*$.
Let $m_1, m_2 \in L_0^{\infty}(w)^*$. Since $\theta$ is an isometrically isomorphism, $\theta(u)=u$.
Thus
\begin{eqnarray*}
d(m_1 m_2) &=& D(um_1 um_2) \\
&=& D(um_1) \theta(um_2) + \theta(um_1) D(um_2) \\
&=& D(um_1) \theta(u) \theta(m_2) + \theta(u) \theta(m_1) D(u m_2) \\
&=& D(um_1) \theta(m_2) + \theta(m_1) D(um_2) \\
&=& d(m_1) \theta(m_2) + \theta (m_1) d(m_2).
\end{eqnarray*}
Hence $d \in \emph{Der}( L_0^{\infty}(w)^*, u L_0^{\infty}(w)^*)$. On the other hand,
\begin{equation*}
u L_0^{\infty}(w)^* \cap \hbox{ran}(L_0^{\infty}(w)^*) = \{0\}.
\end{equation*}
These facts prove the result.
\end{proof}
In the following, let $C_0(w)$ be the Banach space of all
complex-valued continuous functions $f$ on $[0, \infty)$ such that $f/w$ vanishes at infinity.

\begin{proposition} \label{m5}
Let $\theta$ be a homomorphism on $L_0^{\infty} (w)^*$,
$D: L_0^{\infty} (w)^* \rightarrow \emph{ran}(L_0^{\infty}(w)^*)$ be a $\theta$-derivation and $m \in L_0^{\infty} (w)^*$.
If $D(m)$ is positive, then $D(m)=0$.
\end{proposition}
\begin{proof}
Let $m \in L_0^{\infty} (w)^*$ and $D(m)$ be positive. Then
\begin{equation*}
\|D(m)\|= \| D(m)|_{C_0(w)}\| =0;
\end{equation*}
see [8].
\end{proof}
\begin{corollary} \label{m6}
Let $\theta$ be an isomorphism on $L_0^{\infty} (w)^*$ and $D$ be a $\theta$-derivation on $L_0^{\infty} (w)^*$.
If $m \in L_0^{\infty} (w)^*$ such that is positive, then $D(m)=0$.
\end{corollary}
\begin{proof}
Let $\theta$ be an isomorphism on $L_0^{\infty} (w)^*$.
It follows from Theorem 2.1 that $D$ maps
 $L_0^{\infty} (w)^*$ into $\hbox{ran}(L_0^{\infty}(w)^*)$.
Now, apply Proposition 2.5.
\end{proof}

\begin{theorem} \label{m7}
Let $\theta$ be a homomorphism on $L_0^{\infty}(w)^*$ and $D$ be a $\theta$-derivation on $L_0^{\infty}(w)^*$.
Then $D$ is $k$-centralizing if and only if $D$ is $k$-commuting.
In this case, if $\theta$ is an isometrically isomorphism, then
$D$ maps  $L_0^{\infty}(w)^*$ into $u L_0^{\infty}(w)^*$ and
$D(m)= \theta(u) D(m)$ for all $m \in L_0^{\infty}(w)^*$ and
$u \in \Lambda (L_0^{\infty}(w)^*)$.
\end{theorem}
\begin{proof}
Let $D$ be $k$-centralizing.
Then $[ D(m), m^k] \in Z(L_0^{\infty}(w)^*)$
 for all $m \in L_0^{\infty}(w)^*$.
Let $u \in \Lambda (L_0^{\infty}(w)^*)$. Then
\begin{eqnarray*}
[D(m), m^k] &=& [D(m), m^k] u \\
&=& u [D(m), m^k] \\
&= & u D(m) m^k - u m^k D(m) \\
&=&0.
\end{eqnarray*}
Thus $D$ is $k$-commuting.
The converse is trivial.
Since $D(u) \in \hbox{ran}(L_0^{\infty}(w)^*)$,
we have
\begin{equation} \label{E1}
D(u)= [ D(u), u^k ]=0.
\end{equation}
For every $r \in \hbox{ran}(L_0^{\infty}(w)^*)$,
 we have $r+u= (r+u)^k$.
Thus
\begin{eqnarray*}
0&=& [D(r+u), (r+u)^k]\\
& =& [D(r), r+u]\\
&=& D(r) u- u D(r)\\
&=& D(r),
\end{eqnarray*}
because $D(r) \subseteq \hbox{ran} (L_0^{\infty}(w)^*)$.
Hence $D(r)=0$.
Therefore, by \eqref{E1} we obtain
\begin{eqnarray*}
D(m) &=& D(um) + D(m-um) \\
&=& D(um)= D(u) \theta (m) + \theta(u) D(m) \\
&=& \theta (u) D(m).
\end{eqnarray*}
To complete the proof, we only recall that by Lemma 2.1 of [5],
$\theta(u) \in \Lambda (L_0^{\infty}(w)^*)$.
\end{proof}
\vspace{0.5 cm}

\end{document}